\documentclass{amsart}
\usepackage{geometry}                
\usepackage{graphicx}
\usepackage{amssymb}
\usepackage{epstopdf,amsmath}

\newtheorem{prop}{Proposition}[section]
\newtheorem{lemma}[prop]{Lemma}
\newtheorem{cor}[prop]{Corollary}
\newtheorem{theorem}[prop]{Theorem}
\DeclareMathOperator{\Rank}{Rank}
\newtheorem{axiom}{Axiom}

\newcommand{\eps}{\epsilon}

\newcommand{\ZZ}{\mathbb{Z}}
\newcommand{\RR}{\mathbb{R}}

\title{The joints problem for matroids}

\author{Larry Guth and Andrew Suk}

\begin{document}

\maketitle

\begin{abstract} We prove that in a simple matroid, the maximal number of joints that can be formed by $L$ lines is
$o(L^2)$ and $\Omega(L^{2 - \eps})$ for any $\eps > 0$. \end{abstract}

If $\frak L$ is a set of lines in $\RR^3$, a joint of $\frak L$ is a point $x \in \RR^3$ which lies in three non-coplanar lines of $\frak L$.  Using
a grid, it's straightforward to give examples with $L$ lines and $\sim L^{3/2}$ joints.  In the early 90's, Chazelle, Edelsbrunner, Guibas, Pollack, Seidel, Sharir, and Snoeyink defined joints and raised the problem how many joints can be formed by $L$ lines \cite{CEGPSSS}.  They proved that the maximal number of joints is $\le C L^{7/4}$ and conjectured that the maximal number of joints is $\le C L^{3/2}$.  After fifteen years, this conjecture was proven using an unexpected trick with high degree polynomials.
The first proof appeared in \cite{GK1}, and simplified proofs appeared in \cite{KSS} and \cite{Q}.  The simplified proofs are only about one
page long.

The joints problem has a very short proof with high degree polynomials and it seems hard to prove without high degree polynomials.  But it's not obvious what high degree polynomials have to do with the problem.  The problem only involves lines and planes, which are linear objects, but the proof involves highly non-linear polynomials.  In this paper, we try to study why the joints theorem is hard to prove using purely linear tools.

Richard Stanley suggested to us to look at the joints problem for matroids.  The joints problem is about points, lines, and planes in $\RR^3$.  Matroids are generalizations of vector spaces.  They have enough structure to define lines and planes and to set up the joints problem.  The setup works best in a simple matroid, which we will define below.  The lines and planes in a simple matroid obey many standard properties of lines and planes in $\RR^3$.  For example, if a line intersects a plane in at least two points, then the line is contained in the plane.

Our main result says that the joints theorem is false in simple matroids.  In fact, for any $\eps > 0$, we will construct a simple matroid containing a set of $L$ lines that determine $\ge L^{2 - \eps}$ joints.  For each $\eps$, the number $L$ can be made arbitrarily large.

This result helps to explain why the joints theorem is hard to prove without polynomials.  To prove the joints theorem, it is necessary to use
some piece of information which is true in $\RR^3$ and false in other simple matroids.  But most straightforward facts about lines and planes
in $\RR^3$ are true in any simple matroid.  We will give an explicit list of such facts later on.  These facts are not enough to
prove the joints theorem.

It's elementary to check that $L$ lines in a simple matroid determine $\le L^2$ joints - it follows because two lines intersect in at most one point.  We will improve this elementary bound, showing that the number of joints is $o(L^2)$.  So the maximal number of joints in a simple matroid grows more slowly than $L^2$ but faster than $L^{2 - \eps}$ for any $\eps > 0$.

We don't assume that the reader has any familiarity with matroids.  We will recall the definitions and give self-contained proofs of all facts about matroids that we use.  In the next section, after recalling the relevant definitions, we will state our theorems precisely.

Perhaps our results could also be of interest to mathematicians studying matroids.  Matroids are generalizations of vector spaces.  One question in matroid theory is to understand which properties of vector spaces hold more generally for matroids.
Sometimes a theorem about vector spaces holds more generally, and other times there are matroids that behave differently from vector spaces.  Our theorem gives a new example of how matroids can behave differently from vector spaces.

{\bf Acknowledgements.} Francisco Santos greatly simplified the proof of Proposition \ref{matroid} in the first version of this paper.

\section{Background on matroids}

In this section, we give background on matroids and state our results.  First we quickly introduce matroids and give enough definitions to state our theorems.  Then we come back and flesh out the description of matroids.

Suppose that $x_i$ are points in $\RR^n$.  The affine span of the set $\{ x_i \}$ is the intersection of all the affine subspaces containing the points $x_i$.  Algebraically, the affine span of the set $\{ x_i \}$ is the set of points of the form $\sum a_i x_i$, where
$\sum a_i = 1$.  The affine span of a set is always an affine subspace.  A set of $k$ points is affinely independent if its affine span has dimension $k - 1$.  Otherwise, it's affinely dependent.  As we will see, lines, planes, and the joints problem can all be rephrased in terms of the affinely independent sets of $\RR^3$.

A matroid is a pair $(E,\mathcal{I})$, where $E$ is a finite set and $\mathcal{I}$ is a list of ``independent'' subsets of $E$ obeying three axioms:

\begin{axiom}  The empty set is independent. \end{axiom}

\begin{axiom} A subset of an independent set is independent. \end{axiom}

\begin{axiom} If $X_1$ and $X_2$ are independent sets, and $|X_1| < |X_2|$, then there is an element $e \in X_2 \setminus X_1$ so that $X_1 \cup e$ is independent. \end{axiom}

Axioms 1-3 hold for the affine independent sets in $\RR^n$.  Because $\RR^n$ is infinite, the set $\RR^n$ and its affinely independent subsets don't quite make a matroid.  But for any finite $E \subset \RR^n$, the set $E$ and the affinely independent subsets of $E$ make a matroid.  Matroids capture some of the fundamental features of vector spaces, dimensions, etc.  But there are many matroids that don't come from an affine space or vector space.  Using just the structure $(E, \mathcal{I})$ and the three axioms, we can define lines and planes and prove many of the basic properties of lines and planes in $\RR^n$.

Suppose that $(E, \mathcal{I})$ satisfies the three axioms above.   The rank of a subset $Y \subset E$ is defined as the largest cardinality of an
independent set $X \subset Y$.  (In a matroid, since $E$ is finite, every set has a finite rank.)  If $E = \RR^3$ and $\mathcal{I}$ is the affinely independent subsets of $\RR^3$, then a point has rank 1, a line has rank 2,
a plane has rank 3, and the whole space $\RR^3$ has rank 4.  We also note that a line in $\RR^3$ is a maximal set of rank 2 with respect to inclusion: if we add any other point to a line, the rank jumps to 3.  Similarly, a point is
a maximal set of rank 1, a plane is a maximal set of rank 3, etc.

Based on the analogy with affinely independent sets in $\RR^n$, we define an affine $k$-dimensional flat in a matroid
$M$ to be a set of rank $k+1$ which is maximal with respect to inclusion.  For low values of $k$ we will use
simpler words: a point is defined to be a maximal set of rank 1, a line a maximal set of rank 2, and a plane a maximal set of rank 3.

A matroid is called simple if every set of one or two elements is independent.  In a simple matroid, a set $X$ has rank 1 if and only if $X$ consists of a single element.  Therefore, in a simple matroid, the points are exactly the 1-element subsets of $E$.  We think of $E$ as the set of points of the simple matroid.  When we set up the joints problem, we will only work with simple matroids.

We now have enough definitions to set up the joints problem in a simple matroid.  Let $M$ be a simple matroid on a set $E$.  Suppose that
$\frak L$ is a set of lines in the matroid.  We say that some lines $l_1, l_2, l_3$ are coplanar if their union is contained in a plane
of the matroid $M$.   A point $x \in E$ is a joint for $\frak L$ if the point $x$ lies in three non-coplanar lines of $\frak L$.  We can
now pose the question: if $\frak L$ is a set of $L$ lines in a simple matroid, what is the maximal number of joints that $\frak L$ can determine?

As we will see below, two lines in a simple matroid intersect in at most one point.  Therefore, the number of joints formed by $L$ lines is
$\le L^2$.  Our first result slightly improves this trivial bound: $L$ lines in a simple matroid can only determine $o(L^2)$ joints.

\begin{theorem}\label{upperbound}
For any $\epsilon > 0$, there exists an integer $L_0 = L_0(\epsilon)$ such that, in any simple matroid $M$, any set of $L \geq L_0$ lines determines at most $\epsilon L^2$ joints.
\end{theorem}

\noindent Our next result says that this upper bound is nearly tight.

\begin{theorem} \label{mainthm} For any $\eps > 0$, for arbitrarily large numbers $L$, we will construct a simple matroid $M$ and a set of $L$ lines in the matroid which determines $\ge L^{2 - \eps}$ joints.  \end{theorem}

Now that we've stated our theorems, we come back and flesh out the definition of matroids.  Lines and planes in a simple matroid
share many basic properties with lines and planes in $\RR^n$.  Here are some examples.

\begin{prop} \label{incidlp} Let $M$ be a simple matroid.

\begin{enumerate}

\item Any two points are contained in a unique line.

\item If three points are not contained in a line, then they are contained in a unique plane.

\item If a line intersects a plane in two points, then the line is contained in the plane.

\item If two lines, $l_1, l_2$ intersect in a point, then $l_1 \cup l_2$ lies in a unique plane.

\end{enumerate}

\end{prop}

Because of Theorem \ref{mainthm}, these properties of points and lines are not enough to prove that $L$ lines in
$\RR^3$ determine $\lesssim L^{1.99}$ joints.

We now give an outline of the rest of the paper.   In  Section \ref{upperboundsec}, we prove Theorem \ref{upperbound}.
The only results about matroids used in the proof of Theorem \ref{upperbound} are contained in Proposition \ref{incidlp}.
In Section \ref{construction}, we prove Theorem \ref{mainthm}.  The proof uses only the definition of a matroid.

In the rest of this section, we prove some fundamental (classical) results about matroids, building up to Proposition \ref{incidlp}.  Our exposition follows \cite{O}.  Chapter 1
of \cite{O} contains a good introduction to matroids, including all these results and more.  We give a more compact presentation of the particular
results that are relevant in this paper.

A fundamental result about matroids describes how the rank of a union behaves.

\begin{theorem} \label{rankunion} If $X$ and $Y$ are sets in a matroid, then

$$\Rank(X \cup Y) + \Rank(X \cap Y) \le \Rank(X) + \Rank(Y).$$

\end{theorem}

\begin{proof} Let $I_{X \cap Y}$ be an independent subset of $X \cap Y$ with cardinality $\Rank(X \cap Y)$.
The key observation in the proof is that $I_{X \cap Y}$ is contained in an independent set $I_{X \cup Y} \subset
X \cup Y$ with cardinality $\Rank (X \cup Y)$.  By definition, there is an independent set $I'_{X \cup Y} \subset
X \cup Y$ with cardinality $\Rank (X \cup Y)$.  But we don't necessarily have $I_{X \cap Y} \subset I'_{X \cup Y}$.
If $|I_{X \cap Y} | = |I'_{X \cup Y}|$, then we are done.  If not, Axiom 3 tells us that we can find $e_1 \in I'_{X \cup Y}$ so that $I_{X \cap Y} \cup e_1$ is an independent set of cardinality $\Rank(X \cap Y) + 1$.  If $|I_{X \cap Y} \cup e_1| = |I'(X \cup Y)|$, we are done.  If not, Axiom 3 tells us that we can find $e_2 \in I'_{X \cup Y}$ so that
$I_{X \cap Y} \cup e_1 \cup e_2$ is an independent set of cardinality $\Rank (X \cap Y) + 2$.  Continuining in this
way, we build an independent set $I_{X \cup Y}$ with $I_{X \cap Y} \subset I_{X \cup Y} \subset X \cup Y$ and
$|I_{X \cup Y}| = \Rank(X \cup Y)$.

We let $I_{X} := X \cap I_{X \cup Y}$, and we let $I_Y := Y \cap I_{X \cup Y}$.  By Axiom 2, $I_X$ and $I_Y$ are independent, so $|I_X| \le \Rank(X)$ and $|I_Y| \le \Rank (Y)$.  Also, $I_X \cap I_Y$ contains $I_{X \cap Y}$.  Now the rest of the proof is just counting.

$$ \Rank (X \cup Y) + \Rank (X \cap Y) = | I_{X \cup Y} | + | I_{X \cap Y} | \le $$

$$ \le | I_X \cup I_Y | + | I_X \cap I_Y | = |I_X| + |I_Y| \le \Rank(X) + \Rank(Y). $$

\end{proof}

Remark. There are several places in the fundamental theorems where we use the fact that the ground set $E$ is finite.  For example, since $E$ is finite, the rank of a set $X \subset E$ is clearly finite, and we can find an independent set $I \subset E$ with $|I| = \Rank (X)$.  These issues motivate choosing the definition of a matroid so that $E$ is finite.

The following special case of Theorem \ref{rankunion} will be particularly important for us.

\begin{cor} \label{rankunion'} If $X \subset Y_1, Y_2$, and $\Rank(X) = \Rank(Y_1) = \Rank(Y_2)$, then $\Rank (Y_1 \cup Y_2) =
\Rank(X)$.
\end{cor}

\begin{proof} By Theorem \ref{rankunion}, we have

$$ \Rank (Y_1 \cup Y_2) + \Rank (X) \le \Rank(Y_1 \cup Y_2) + \Rank (Y_1 \cap Y_2) \le \Rank(Y_1) + \Rank(Y_2)
= 2 \Rank (X). $$

\end{proof}

For any set $X \subset E$, let $\{ Y_j \}$ be all the sets containing $X$ with $\Rank (Y_j) = \Rank (X)$.  Using Corollary \ref{rankunion'} repeatedly, we see that $\Rank( \cup_j Y_j) = \Rank (X)$.  We define the closure of $X$ to be this union: $Cl(X) = \cup_j Y_j$.  We summarize this information in the following corollary.

\begin{cor} \label{closure} For any matroid $(E, \mathcal{I})$, for any $X \subset E$, the closure of $X$ obeys the following properties.

\begin{itemize}

\item $X \subset Cl(X)$.

\item $\Rank(X) = \Rank (Cl(X))$.

\item If $X \subset Y$ and $\Rank(Y) = \Rank(X)$, then $Y \subset Cl(X)$.

\end{itemize}

\end{cor}

(If $X$ is a set in $\RR^n$ with the infinite matroid of affinely independent sets, then $Cl(X)$ is the affine span
of $X$.)

A set $X$ is called a flat if $Cl(X) = X$.  The closure of any set is a flat, by the following lemma.

\begin{lemma} For any set $X$, $Cl(Cl(X)) = Cl(X)$.
\end{lemma}

\begin{proof} Clearly $X \subset Cl(X) \subset Cl(Cl(X))$.  On the other hand, we know that $\Rank Cl( Cl(X)) =
\Rank Cl(X) = \Rank X$.  By Corollary \ref{closure}, $Cl(Cl(X)) \subset Cl(X)$.   \end{proof}

Earlier, we discussed sets of a given rank that are maximal with respect to inclusion.  A flat is exactly such a set.

\begin{lemma} \label{flatmax} A set $X$ is a flat if and only if $X$ is a set of rank $\Rank(X)$ which is maximal with respect to inclusion.
\end{lemma}

\begin{proof} Suppose that $X$ is a flat.  In other words, $Cl(X) = X$.  Let $X$ be a proper subset of $X'$.  We have to show that
$\Rank (X') > \Rank (X)$.  But if $\Rank(X') = \Rank (X)$, then Corollary \ref{closure} implies that $X' \subset Cl(X) = X$.

Now suppose that $X$ is a maximal set of rank $\Rank(X)$.  We have $X \subset Cl(X)$ and $\Rank(Cl(X)) = \Rank (X)$.  By maximality,
we must have $Cl(X) = X$.  Then $X$ is a flat. \end{proof}

Recall that we defined an affine
$k$-dimensional flat of a matroid to be a maximal set of rank $k+1$.  By Lemma \ref{flatmax}, an ``affine $k$-dimensional flat" is just a flat of rank $k+1$.  In particular, a point is a rank 1 flat, a line is a rank 2 flat, and a plane is a rank 3 flat.  (Technical point.  The empty set is
a flat by our definition.  In a simple matroid, the empty set is the unique flat of rank 0.)  The language of flats and closures
will be useful for understanding lines and planes in a matroid.

\begin{lemma} \label{uniquespan} For any $k$, a set of rank $k$ is contained in a unique rank $k$ flat.
\end{lemma}

\begin{proof} Suppose $X$ has rank $k$.  Then $Cl(X)$ is a flat of rank $k$.  Suppose that $X \subset F$
a flat of rank $k$.  Since $\Rank F = k$, $F \subset Cl(X)$.  Since $F$ has rank $k$ and $Cl(X)$ has rank $k$, we
have $Cl(X) \subset Cl(F) = F$.  Hence $F = Cl(X)$.  \end{proof}

\begin{lemma} \label{closincl} If $X \subset Y$, then $Cl(X) \subset Cl(Y)$.
\end{lemma}

\begin{proof} Clearly $Y \subset Cl(X) \cup Y$.  We will check that $\Rank (Cl(X) \cup Y) = \Rank (Y)$.  Then
by Corollary \ref{closure}, it follows that $Cl(X) \cup Y \subset Cl(Y)$.  In particular, this will show that $Cl(X)
\subset Cl(Y)$.  Clearly $\Rank (Cl(X) \cup Y) \ge \Rank(Y)$.  So it only remains to check that $\Rank (Cl(X) \cup Y)
\le \Rank (Y)$.  To check this, we use Theorem \ref{rankunion}.

$$ \Rank ( Cl(X) \cup Y) + \Rank(X) \le \Rank( Cl(X) \cup Y) + \Rank (Cl(X) \cap Y) \le $$

$$ \le \Rank( Cl(X)) + \Rank (Y) = \Rank(X) + \Rank(Y). $$

Subtracting $\Rank(X)$ from both sides gives the estimate. \end{proof}

We now prove that the intersection of two flats is a flat, as for flats in $\RR^n$.

\begin{theorem} \label{flatint} If $F_1$ and $F_2$ are flats in a matroid, then $F_1 \cap F_2$ is also a flat.
\end{theorem}

\begin{proof} To prove that $F_1 \cap F_2$ is a flat, we have to check that $Cl(F_1 \cap F_2) = F_1 \cap F_2$.
Now for any set $X$, $X \subset Cl(X)$, so we just have to show that $Cl(F_1 \cap F_2) \subset F_1 \cap F_2$.

Clearly $F_1 \cap F_2 \subset F_1$.  By Lemma \ref{closincl}, $Cl(F_1 \cap F_2) \subset Cl(F_1) = F_1$.
Similarly, $Cl(F_1 \cap F_2) \subset Cl(F_2) = F_2$.  Therefore, $Cl(F_1 \cap F_2) \subset F_1 \cap F_2$. \end{proof}

Here is another simple fact about flats.

\begin{prop} \label{flatinflat} Suppose that $F_1 \subset F_2$ are flats in a matroid.  Then either $F_1 = F_2$ or $\Rank(F_1) <
\Rank(F_2)$.
\end{prop}

\begin{proof} Suppose $\Rank (F_1) = \Rank (F_2)$.  Since $F_1 \subset F_2$, we have $F_2 \subset Cl(F_1) =
F_1$.  \end{proof}

Now we have enough background knowledge to quickly prove Propositon \ref{incidlp}.

\begin{proof} (1) Let $X$ be a set of 2 points.  In a simple matroid, any set of 2 points is independent, so $Rank(X) = 2$.  Therefore $Cl(X)$ is a rank 2 flat, which is a line.  Now let $l_1$ and $l_2$ be two lines containing $X$.
 The rank of $l_1 \cap l_2$ is at least the rank of $X$ which is 2 and at most the rank of
$l_1$ which is 2.  By Theorem \ref{flatint}, $l_1 \cap l_2$ is a flat containing $X$.  In short, $l_1 \cap l_2$ is a rank 2
flat.  Since $l_1 \cap l_2 \subset l_1$, Proposition \ref{flatinflat} implies that $l_1 \cap l_2 = l_1$.  Similarly,
$l_1 \cap l_2 = l_2$.  Therefore, $l_1 = l_2$.  This shows that $X$ is contained in a unique line.

(2) Let $X$ be a set of 3 points not contained in any line.  If $X$ had rank $2$, then $X$ would be contained in a maximal rank 2 set, which is a line.
Therefore, $X$ has rank 3.  By Lemma \ref{uniquespan}, $X$ lies in a unique rank 3 flat.  In other words, $X$ lies in a unique plane.

(3) Let $l$ be a line and let $\pi$ be a plane in a simple matroid, and suppose that $l \cap \pi$ contains at least two points.
By Theorem \ref{flatint}, we know that $l \cap \pi$ is a flat.  Since the matroid is simple,
the rank of $l \cap \pi$ is at least 2.  On the other hand, $l \cap \pi \subset l$, so it has rank $\le 2$.  In short $l \cap \pi$ is a rank 2
flat.   So $l \cap \pi \subset l$ are both rank 2 flats.  By Proposition \ref{flatinflat}, $l \cap \pi = l$.  Hence $l \subset \pi$.

(4) Let $l_1$ and $l_2$ be two lines in a simple matroid that intersect at a point $p$.  Since $l_2$ has rank 2, it must contain some point $p_2 \not= p$.
By (1) above, $l_1 \cap l_2$ consists of $\le 1$ point, and so $p_2 \notin l_1$.  Similarly, we can find a point $p_1 \in l_1 \setminus l_2$.  We claim that
the three points $p_1, p_2, p$ do not all lie in a line.  They don't all lie in $l_1$.  Any other line intersects $l_1$ in at most one point, so no other line
contains both $p$ and $p_1$.  By (2) above, $p, p_1, p_2$ lie in a unique plane.  By (3), $l_1$ and $l_2$ also lie in this plane.

\end{proof}

We have now covered all the results about matroids that we will use in the sequel, and hopefully given a little flavor for matroids.

\section{Upper bound on the number of joints} \label{upperboundsec}

In this section we prove Theorem \ref{upperbound}.  The main tool of the proof is the following theorem of Ruzsa and Szemeredi \cite{ruzsa}, which is known in the literature as the triangle removal lemma (see also \cite{fox}).

\begin{lemma}
\label{triangle}
Let $G$ be a graph with vertex set $V$.  If $G$ contains $\epsilon |V|^2$ edge-disjoint triangles, then $G$ contains at least $\delta |V|^3$ triangles, where $\delta$ depends only on $\epsilon$.
\end{lemma}

We will also use the properties of lines and planes in a simple matroid given in Proposition \ref{incidlp}.

\vskip5pt

\noindent \emph{Proof of Theorem \ref{upperbound}.}  Let $\epsilon > 0$ and $M  = (E,\mathcal{I})$ be a simple matroid with $L \geq L_0$ lines, where $L_0 = L_0(\epsilon)$ will be determined later.  Let $\frak L$ be a set of lines in $M$.  For the sake of contradiction, suppose that $\frak L$ determines more than  $\epsilon L^2$ joints in $M$.

As long as there is a plane $h$ containing $L_h \geq 2/\epsilon$ lines, we remove from $\frak L$ the $L_h$ lines.  By Proposition \ref{incidlp}, we know that each line not contained in $h$ meets $h$ in at most one point, which implies that $h$ contains at most $L - L_h \leq L$ joints.  Therefore removing all lines contained in $h$ removes at most $L$ joints.  The number of planes $h$ considered is at most $\epsilon L/2$, which implies that at most $\epsilon L^2/2$ joints are removed in this process.  Let $\frak L'$ be the set of remaining lines, which forms at least $\epsilon L^2/2$ joints.  No plane contains
$> 2 / \eps$ lines of $\frak L'$.

For each $x \in E$, let $d(x)$ denote the number of lines in $\frak L'$ that contain $x$.  Then we define

$$E_1 = \{x \in E: d(x)\geq 4/\epsilon\}\hspace{.5cm}\textnormal{and}\hspace{.5cm}E_2=\{x \in E: 3\leq d(x) < 4/\epsilon\}.$$

\noindent  By Proposition \ref{incidlp}, every pair of lines in $\frak L'$ have at most one point in common, and so we have

$$|E_1|\frac{4}{\epsilon}\leq \sum\limits_{x\in E_1}d(x) \leq \sum\limits_{x \in E_1\cup E_2} d(x) \leq |\frak L'|^2\leq  L^2.$$

\noindent  Hence $|E_1|\leq \epsilon L^2/4$ and therefore $|E_2| \geq \epsilon L^2/4$.  (By a similar argument, $|E_2| \le
L^2$.)

Now we define that graph $G$ whose vertex set is $\frak L'$ and two vertices are adjacent in $G$ if and only if the corresponding lines intersect at a point from $E_2$.  Note that $G$ has at most $L$ vertices.  Since each point in $E_2$ is a joint, this implies that $G$ contains at least $\epsilon L^2/4$ edge-disjoint triangles.  By Lemma \ref{triangle}, $G$ contains at least $\delta L^3$ triangles, where $\delta$ depends only on $\epsilon$.

We say that $l_1,l_2,l_3 \in \frak L'$ form a \emph{degenerate} triangle in $G$, if there exists a point $x\in E_2$ such that $l_1\cap l_2\cap l_3 = x$.  Since $d(x) \leq 4/\epsilon $ for every $x \in E_2$, the number of degenerate triples in $G$ is at most

$${4/\epsilon \choose 3}|E_2| \leq 4^3 \eps^{-3} L^2.$$

\noindent For $L$ sufficiently large we have $4^3 \eps^{-3} L^2 < (\delta/2)L^3$, and therefore $G$ contains at least $\delta L^3/2$ non-degenerate triangles.  Notice that if $l_1,l_2,l_3$ forms a non-degenerate triangle in $G$, then there are distinct points $x_1,x_2,x_3 \in E_2$ such that $l_1\cap l_2 = x_1, l_2\cap l_3 = x_2$, and $l_1\cap l_3 = x_3$.

Since $G$ contains at least $\delta L^3/2$ non-degenerate triangles, we can choose two lines $l_1,l_2\in \frak L'$ that participate in at least $\delta L/2$ non-degenerate triangles.  In other words, there are $\delta L / 2$ lines $l \in \frak L '$ so that $l_1, l_2, l$ form a non-degenerate triangle.  In order to participate in a non-degenerate triangle, $l_1$ and $l_2$
must intersect in a point of $E_2 \subset E$.
Now by Proposition \ref{incidlp}, $l_1 \cup l_2$ lies in a unique plane $\pi$.  Suppose that $l_1, l_2, l$ form a non-degenerate triangle.  Then $l$ must intersect
$l_1 \cup l_2$ at two distinct points.  So $l$ intersects $\pi$ at two distinct points.  By Proposition \ref{incidlp}, $l$ lies in the plane $\pi$.  Therefore,
$\pi$ contains $\ge \delta L / 2 + 2$ lines of $\frak L'$.   For sufficiently large $L > L_0(\epsilon) \geq 100/(\epsilon \delta)$ we have

$$\frac{\delta}{2}L + 2 > \frac{2}{\epsilon},$$

\noindent which is a contradiction since no plane contains more than $2/\epsilon$ lines from $\frak L'$.  This completes the proof of Theorem \ref{upperbound}. $\hfill\square$

\section{Constructing matroids where $L$ lines can make $L^{2-\epsilon}$ joints}
\label{construction}
The construction of our matroids will be based on configurations of lines and points in $\RR^n$.  Suppose that
$E$ is a finite set of points in $\RR^n$ and $\frak L$ is a finite set of lines in $\RR^n$.  A \emph{triangle} in $(E, \frak L)$ will
mean a set of three distinct points $x_1, x_2, x_3 \in E$ and three distinct lines $l_1, l_2, l_3 \in \frak L$ so that each line $l_i$ contains exactly two of the points $x_j$.  We say that $(E, \frak L)$ is triangle free if there are no triangles in $(E,
\frak L)$.

If $(E, \frak L)$ is triangle free, then we will use $\frak L$ to construct a matroid on the set $E$ with some good properties.

\begin{prop} \label{matroid} Suppose that $(E, \frak L)$ is triangle free and that each line of $\frak L$ contains at least two points of
$E$.  Then there is a simple matroid $M$ on $E$ with the following properties.

\begin{enumerate}

\item For each line $l \in \frak L$,  $l \cap E$ is a line in the matroid $M$.

\item If $x \in E$ and $l_1, l_2, l_3 \in \frak L$ are lines containing $x$, then $E \cap (l_1 \cup l_2 \cup l_3)$ is not contained in any plane of the matroid $M$.

\item The matroid $M$ has rank at most 4.

\end{enumerate}

\end{prop}

By abuse of notation, we can think of $\frak L$ as a set of $L$ lines in the matroid $M$ on the set $E$.  If $x \in E$
lies in three lines of $\frak L$, then by the second property, $x$ is a joint of $\frak L$ in the matroid $M$.

A crucial point here is that three lines of $\frak L$ may be coplanar in $\RR^n$ but not lie in any plane in the matroid $M$.
Therefore, a point $x \in E$ may not be a joint for the lines $\frak L \subset \RR^n$, but may still be a joint for the lines
$\frak L$ in the matroid $M$ on $E$.

Based on a construction by Behrend \cite{Be} and Ajtai and Szemer\'edi \cite{ajtai}, we construct examples of $(E, \frak L)$ which are triangle free but still have many triple intersection points.  (Recall that a point $x \in E$ is called a triple intersection point (for $\frak L$) if $x$ lies in three distinct lines of $\frak L$.)

\begin{prop} \label{funnyset} For any $\eps > 0$, and for arbitrarily large $L$, we can find a set $\frak L$ of $L$ lines in $\RR^2$ and a
set $E \subset \RR^2$ with the following properties.

\begin{enumerate}

\item The pair $(E, \frak L)$ is triangle free.

\item Each line of $\frak L$ contains at least two points of $E$.

\item The number of triple intersection points in $E$ is $\ge L^{2 - \eps}$.

\end{enumerate}

\end{prop}

Theorem \ref{mainthm} follows immediately from these two Propositions.  Let $(E, \frak L)$ be the points and lines given in
Proposition \ref{funnyset}.  By Proposition \ref{matroid}, we can find a matroid structure $M$ on $E$ so that each line of $\frak L$ corresponds to a line of $M$.  Each triple intersection point of $(E, \frak L)$ corresponds to a joint of the lines $\frak L$ in the matroid $M$.  So in this simple matroid, we have a set of $L$ lines that determines $\ge L^{2 - \eps}$ joints.  Incidentally, the rank of $M$ is 4, the same as the rank of $\RR^3$ equipped with affine independent subsets.

\subsection{Matroids from triangle free configurations}

In this subsection, we prove Proposition \ref{matroid}.  The proof below is due to Francisco Santos.  It greatly simplifies our original argument.  

Suppose that $E \subset \RR^n$ and $\frak L$ is a set of lines in
$\RR^n$.  Suppose that $(E, \frak L)$ is triangle free, and that each line of $\frak L$ contains at least two points of $E$.
We have to construct a simple matroid on $E$ with some good properties.  To do this, we list the independent and dependent sets
of the matroid.  Then we will check that they obey the axioms of a simple matroid.

The empty set is independent.  Any set with one or two points is independent.

A set with 3 points is dependent if and only if all three points lie on a line $l \in \frak L$.

A set $X$ with 4 points is dependent if either of the following occurs:

\begin{enumerate}

\item $X$ contains three points which lie on a line $l \in \frak L$.

\item There are lines $l, l' \subset \frak L$ so that $X \subset l \cup l'$ and $l \cap l' \cap E$ is non-empty.

\end{enumerate}

Any set with more than 4 points is dependent.

If $l, l' \subset \frak L$ and $l \cap l' \cap E$ is non-empty, we call $l \cup l'$ an angle.  So a set of 4 points is dependent if either three of the points lie in a line of $\frak L$ or all of the points lie in an angle. 

\begin{prop} If $(E, \frak L)$ is triangle free, then this list of independent sets obeys the axioms of a simple matroid.
\end{prop}

\begin{proof} Most of the axioms of a matroid can be dealt with quickly.  At the end, there will be one more complex case where
we use that $(E, \frak L)$ is triangle free.

Axiom 1. The empty set is independent.  This is immediate from the definition.

Axiom 2. Any subset of an independent set is independent.  Suppose that $X$ is independent and $X' \subset X$.  We can assume that $X'$ is a proper subset of $X$.  If $X'$ has $\le 2$ points, then $X'$ is independent.  Otherwise, $X'$ must contain
3 points and $X$ must contain 4 points.  Since $X$ is independent, we see that the points of $X'$ are not all on a line of $\frak L$, which means that $X'$ is independent.

Axiom 3. If $X_1$ and $X_2$ are independent sets with $|X_1| < |X_2|$, then there exists $e \in X_2 \setminus X_1$ so that
$X_1 \cup e$ is independent.

We begin with the case that $|X_1| \le 2$, which is the easier case.  If $|X_1| \le 1$, we can take $e$ to be any element of $X_2 \setminus X_1$.  Then $X_1 \cup e$ has at most two elements and is independent.  Now suppose that $|X_1| = 2$.  If $X_1$
is not contained in a line $l \in \frak L$, then again we can take $e$ to be any element of $X_2 \setminus X_1$.  Then $X_1 \cup e$ will be independent because $X_1 \cup e$ will be a set of 3 points which don't all lie on a line of $\frak L$.  Now suppose that
$|X_1| = 2$ and $X_1$ is contained in a line $l \in \frak L$.  This line $l$ must be unique, because two lines intersect in at most
one point.  Since $|X_2| \ge 3$, $X_2$ is not contained in $l$.  We let $e$ be an element of $X_2 \setminus l$.  We see that
$e \in X_2 \setminus X_1$ and that $X_1 \cup e$ is independent.

We are left with only one case: $|X_1| = 3$ and $|X_2| = 4$.  In some sense, this is the main case.  

Suppose that $X_1 = \{ a, b, c \}$.  
Since $X_1$ is independent, the points $a,b,c$ don't all lie on a line of $\frak L$.  We analyze several cases depending on how many pairs of the points of $X_1$ lie on lines of $\frak L$.  Because $(E, \frak L)$ is triangle free, the number of these pairs is 0, 1, or 2.

{\bf Case 0.} Suppose that no pair of $a,b,c$ lies on a line of $\frak L$.  Then let $e$ be any point of $X_2 \setminus X_1$.  Now $X_1 \cup e$ is independent, because no line of $\frak L$ contains three points of $X_1 \cup e$ and no two lines of $\frak L$ contain
$X_1 \cup e$.

{\bf Case 1.} Suppose that $a$ and $b$ lie in a line $l \in \frak L$, and no other pair of points in $X_1$ lies in a line of $\frak L$.  We claim that $X_1$ lies in at most one angle.  Suppose that $l_1 \cup l_2$ is an angle containing $X_1$.  By relabelling, we can assume that $l_1$ contains two points of $X_1$, and so $l_1 = l$.  Now it follows that $l_2$ contains $c$, and $l_2 \cap l \cap E$ is non-empty.  Since $(E, \frak L)$ is triangle free, there is at most one such line $l_2$.  So there is at most one angle containing $X_1$.

If $X_1$ lies in an angle, we can choose a point $e \in X_2$ which doesn't lie in that angle.  In particular, $e \notin X_1$ and $e \notin l$.  Then we claim that $X_1 \cup e$ is independent.  The only line of $\frak L$ that contains $\ge 2$ points of $X_1$ is $l$.  Since $e \notin l$, no line contains three points of $X_1 \cup e$.  Also, there is only one angle containing $X_1$ and $e$ is not in the angle.  

If $X_1$ does not lie in any angle, then we choose $e$ as follows.  We know that $l$ contains $\le 2$ points of $X_2$, so we can choose $e \in X_2$ with $e \notin l$ and $e \not= c$.  Hence $e \notin X_1$.  We claim that $X_1 \cup e$ is independent.  As above, 
the only line of $\frak L$ that contains $\ge 2$ points of $X_1$ is $l$.  Since $e \notin l$, no line contains three points of $X_1 \cup e$.  Clearly $X_1 \cup e$ is not contained in any angle.

{\bf Case 2.} Suppose that $a$ and $b$ lie in $l \in \frak L$ and $a,c$ lie in $l' \in \frak L$.  Note that $l \cap l' \cap E$ is non-empty: it contains $a$.  Therefore, $l \cup l'$ is an angle.  Since $X_2$ is independent, $X_2$ is not contained in $l \cup l'$.  We choose $e \in X_2$ with $e \notin l \cup l'$, and therefore $e \notin X_1$.  We claim again that $X_1 \cup e$ is independent.  The only lines that contain $\ge 2$ points of $X_1$ are $l$ and $l'$.  Since $e \notin l \cup l'$, no line of $\frak L$ contains three points of $X_1 \cup e$.  

Now suppose that $X_1 \cup e$ is contained in an angle $l_1 \cup l_2$.  By relabelling $l_1$ and $l_2$, we can assume that $l_1$ contains at least two points of $X_1$, and so $l_1$ is $l$ or $l'$.  By relabelling the points in $X_1$, we can assume that $l_1 = l$.  So we know that $l_2$ and $l$ make an angle.  Now $l_2$ contains $e$, but $l'$ does not contain $e$.  Therefore, $l_2 \not= l'$, and $l, l', l_2$ are three distinct lines.  But $l_2$ and $l'$ both contain $c$.  Therefore, $l_2$ and $l'$ make an angle.  In summary $l_2, l,$ and $l'$ are three distinct lines, and any two of them make an angle.  Then $l, l', l_2$ make a triangle.  This contradiction shows that $X_1 \cup e$ is independent.

We have now checked that our definition of independent sets of $E$ obeys the axioms of a matroid. Finally, any set of one or two points is
independent, so our matroid is simple.  \end{proof}

For each triangle free configuration $(E, \frak L)$, we have defined a matroid $M(E, \frak L)$.  We can now finish the proof of Proposition \ref{matroid}.  We suppose that $(E, \frak L)$ is triangle free and that each line of $\frak L$ contains at least two
points of $E$.  We just need to check that the matroid $M(E, \frak L)$ has the desired properties.

Property 1. For each line $l \in \frak L$,  $l \cap E$ is a line in the matroid $M$.

We have to check that $l \cap E$ is a maximal rank 2 set in the matroid $M$.  By our definition of dependence, any three points
on $l \cap E$ are a dependent set.  Therefore, the rank of $l \cap E$ is at most 2.  We know that $l \cap E$ contains two points,
and any set of two points is independent.  Therefore, the rank of $l \cap E$ is exactly two.  Now suppose that $e \in E \setminus l$ -- we have to check that the rank of $(l \cap E) \cup e$ is 3.  Let $X$ be the union of $e$ and two points of $l \cap E$.  We claim that $X$ does not lie in any line $l' \in \frak L$, and so $X$ is independent.  Since $e \notin l$, $X$ is not contained in $l$.
If $l \not= l'$, then $l'$ can only contain one point of $l$, and so $X$ is not contained in $l'$.  Therefore, $l \cap E$ is a maximal
rank 2 set in our matroid.

Property 2. If $x \in E$ and $l_1, l_2, l_3 \in \frak L$ are lines containing $x$, then $E \cap (l_1 \cup l_2 \cup l_3)$ is not contained in any plane of the matroid $M$.

Each line of $\frak L$ contains at least two points of $E$.  Let $a_i$ be a point of $l_i \setminus \{ x \}$ for $i = 1, 2, 3$.  Let $X$ be the set $\{ x, a_1, a_2, a_3 \} \subset E \cap (l_1 \cup l_2 \cup l_3)$.  It suffices to prove that $X$ is independent.  Since $(E, \frak L)$ is triangle free, no line of $\frak L$ contains any two of the points $a_1, a_2, a_3$.  Therefore, no line of $\frak L$ can contain three points of $X$.  Also, no two lines of $\frak L$ can contain $X$.  Therefore, $X$ is independent.

3. The matroid $M$ has rank at most 4.

This follows immediately because every set of 5 points is dependent.

This finishes the proof of Proposition \ref{matroid}.

\subsection{Triangle free configurations}

In this subsection, we prove Proposition \ref{funnyset}.  We produce a configuration of points and lines in $\RR^2$ with no
triangles but many triple intersection points.

We begin with a grid of horizontal, vertical, and diagonal lines.  We call this set of lines $\frak L_0$, and the final set $\frak L$ will
be a subset of $\frak L_0$.  The set $\frak L_0$ consists of the following lines:

\begin{itemize}

\item Horizontal lines $y = b$ for each integer $b = 1, ..., N$.

\item Vertical lines $x = a$ for each integer $a = 1, ..., N$.

\item Diagonal lines $x - y = c$ for each integer $c = - N, ..., N$.

\end{itemize}

The number of lines of $\frak L_0$ is $L_0 = 4 N+1$.

Next we consider the set of points $E$.  We let $E_0$ be the grid of integer points $(a,b)$ with $1 \le a,b \le N$.  We note that each point
of $E_0$ is a triple intersection point for $\frak L_0$.  There are $N^2 \sim L_0^2$ points in $E_0$.  However, $(E_0, \frak L_0)$
has many triangles.  We will prune the set $E_0$ to get rid of the triangles.  Remarkably, there is a subset
$E \subset E_0$ of size $\sim N^{2-\eps}$ so that $(E, \frak L_0)$ is triangle free!  This is the heart of the proof.

This argument is based on Behrend sets.  Behrend was interested in subsets of the integers $1... N$ with no 3-term arithmetic progressions.  (Recall that a 3-term arithmetic progression is just a sequence $a, a+d, a+ 2d$, where $a, d$ are real numbers.)  How
large is the largest subset of $1 ... N$ with no 3-term arithmetic progression?  Behrend gave remarkably large examples.

\begin{theorem}[Behrend, \cite{Be}] For any $\epsilon > 0$, for any sufficiently large $N$, there is a set $B \subset 1 ... N$ so that
$B$ has no 3-term arithmetic progression and $|B| > N^{1 - \eps}$.
\end{theorem}

We will explain Behrend's construction in Section \ref{Behrendapp}.

We let $B \subset 1 ... N$ be a Behrend set, and we use it to define $E$ as follows:

$$ E := \{ (a,b) \in \ZZ^2 | 1 \le a, b \le N \textrm{ and } a + b \in B \}. $$

The pair $(E, \frak L_0)$ still has many triple intersection points, and we will see that it is triangle free.
We would like to prove that $(E, \frak L_0)$ has $\ge N^{2 - \eps}$ triple intersection points.  But since $\eps$ is arbitrary, it's enough to
prove a weaker estimate like $\ge (1/20) N^{2 - 2 \eps}$.  For any subset $D \subset 1 ... N$, the set $\{ (a,b) | 1 \le a, b \le N \textrm{ and }
a + b \in D \}$ has $\ge (1/2) |D|^2$ elements.  The worst case occurs when $D$ is the first $|D|$ integers, and then the set is a lower left corner of
the square.  In particular $|E| \ge (1/2) N^{2 - 2 \eps}$.   Every point of $E$ is a triple intersection point for $\frak L_0$.

\begin{lemma} The pair $(E, \frak L_0)$ is triangle free.
\end{lemma}

\begin{proof} Suppose that $l_1, l_2, l_3$ are lines of $\frak L_0$ forming a triangle.  No two of these lines are parallel, so there must be
one horizontal line, one vertical line, and one diagonal line.  We label them so that $l_1$ is horizontal, $l_2$ is diagonal, and $l_3$ is vertical.  Let $x_1 = (a_1, b_1)$ be the intersection of $l_2$ with $l_3$, and $x_2 = (a_2, b_2)$ be the intersection of $l_1$ and $l_3$ and $x_3 = (a_3, b_3)$ be the
intersection of $l_1$ and $l_2$.  We have $x_1, x_2, x_3 \in E$,
and so $a_1 + b_1, a_2 + b_2, a_3 + b_3 \in B$.  But we claim that the geometry of the situation forces $a_1 + b_1, a_2 + b_2,
a_3 + b_3$ to be a 3-term arithmetic progression.  This contradiction will prove the lemma.

The reader may want to draw a picture to check this.  We give an algebraic proof as follows.  The points $x_1, x_2$ are on the same vertical line $l_3$ and so $a_1 = a_2$.  Next the points $x_2, x_3$ are on the
same horizontal line $l_1$, and so $b_2 = b_3$.  Finally, the points $x_1, x_3$ are on the same diagonal line, and so $a_1 - b_1 =
a_3 - b_3$.  Using these equations, we want to check that $a_1 + b_1, a_2 + b_2, a_3 + b_3$ forms a 3-term arithmetic progression.  This boils down to checking

$$[a_3 + b_3] - [a_2 + b_2] = [a_2 + b_2] - [a_1 + b_1]. $$

Using the equations:

$$ [a_3 + b_3] - [a_2 + b_2] = a_3 - a_2 = a_3 - a_1 = b_3 - b_1 = b_2 - b_1 = [a_2 + b_2] - [a_1 + b_1]. $$

\end{proof}

The pair $(E, \frak L_0)$ has many triple intersection points and no triangles.  The rest of the proof is minor.  We also want to know that each
line of $\frak L$ contains at least two points of $E$.  Some lines of $\frak L_0$ contain no point of $E$ or only one point of $E$.
We define $\frak L \subset \frak L_0$ to be the set of lines in $\frak L_0$ containing at least two points of $E$.

The pair $(E, \frak L)$ is still triangle free.  It may have fewer triple points, but not by much.  The number of points of $E$ contained
in a line of $\frak L_0 \setminus \frak L$ is at most $L_0 = 4 N + 1$.  So the number of triple points of $(E, \frak L)$ is at least
$(1/2) N^{2 - 2 \eps} - (4N + 1)$.  For $N$ sufficiently large this is $\ge (1/20) N^{2 - 2 \eps}$.  Since this holds for every
$\eps > 0$, the number of triple points of $E$ is also $\ge N^{2 - \eps}$.

This finishes the proof of Proposition \ref{funnyset}.

\section{Open problems}

The joints theorem was generalized to higher dimensions by \cite{KSS} and \cite{Q}.  Suppose that $\frak L$ is a set of lines in $\RR^n$.  A point
$x \in \RR^n$ is called an $n$-dimensional joint if $x$ lies in $n$ lines of $\frak L$ which are not contained in any $(n-1)$-dimensional plane.  Kaplan-Sharir-Shustin
and Quilodr\'an proved the following sharp estimate for $n$-dimensional joints.

\begin{theorem} For each dimension $n \ge 3$, there is a constant $C_n$ so that following holds.  Any set of $L$ lines in $\RR^n$ determines at most
$C_n L^{\frac{n}{n-1}}$ $n$-dimensional joints.
\end{theorem}

(The result is also true for $n=2$.  If $n=2$ the result follows immediately from the fact that two lines intersect in at most one point.)

We can make a matroid version of the higher-dimensional joints problem as follows.  Let $M = (E, \mathcal{I})$ be a simple matroid.  Let
$\frak L$ be a set of lines in $M$.  A point $x \in E$ is an $n$-dimensional joint of $\frak L$ if there are $n$ lines $l_1, ...,
l_n \in \frak L$ so that $x \in l_i$ for each $i$ and the rank of $\cup_{i=1}^n l_i$ is $\ge n+1$. Now fix $n \ge 3$.  For a given $L$, what is the maximum number of $n$-dimensional joints that can be formed by $L$ lines in a simple matroid?

For $n=3$, the theorems in this paper give fairly close upper and lower bounds.  If $n \ge 4$, an $n$-dimensional joint is a special case of a 3-dimensional joint.  By Theorem \ref{upperbound}, the number of $n$-dimensional joints is $o(L^2)$ for every $n \ge 3$.  (But not for $n=2$.)  Our examples only give
3-dimensional joints, so we don't have interesting lower bounds.  Can $L$ lines determine $L^{2 - \eps}$ $n$-dimensional joints for larger $n$?

The paper \cite{CEGPSSS} proves that $L$ lines in $\RR^3$ determine $\le C L^{7/4}$ joints.  The proof is based on reguli.  Reguli are degree 2 algebraic surfaces that have a special relationship to lines in $\RR^3$.  Between 1992 and 2008, mathematicians were trying to prove the joints theorem using reguli, but without using high degree algebraic surfaces.  It seems to be difficult to get a sharp exponent by this approach.  It would be interesting to understand why the joints theorem is hard to prove using only lines, planes, and reguli.  Our paper doesn't address this question because matroids don't contain
reguli.  It might be interesting to axiomatize the properties of lines, planes, and reguli, and to see what estimate in the joints problem follows in those
axioms.

\section{Appendix: Behrend sets} \label{Behrendapp}

Here we provide the proof of Behrend \cite{Be}, showing that there are indeed large subsets of $1...N$ with no 3-term arithmetic progressions.  All logarithms are in base 2.  Let $N$ be given and large, and let $n$ and $s$ be integer parameters which will be specified later.  Set $\mathcal{G} = \{0,1,...,s-1\}^n \subset \mathbb{R}^n$ and $S_k = \{x \in \mathcal{G}: ||x||^2 = k\}$, where $||.||$ denotes the Euclidean norm.  Since $\mathcal{G} = \bigcup_{k = 0}^{n(s-1)^2} S_k$, by the pigeonhole principle there exists $k$ such that $|S_k| \geq s^{n-2}/n$.  Since the points of $S_k$ lie on a sphere, no three members in $S_k$ are collinear. We let $n = \lfloor \sqrt{\log N}\rfloor$, and let $s$ be the largest integer such that $(2s)^{n} \leq N$.  Roughly speaking, $s$ will be about $2^{\sqrt{\log N}}/2$.

Now let

$$B = \left\{\sum\limits_{i = 1}^n x_i(2s)^{i-1}: (x_1,x_2,...,x_n) \in S_k\right\}.$$

\noindent Clearly $B$ is a subset of $1...N$ since all elements in $B$ are at most $(2s)^n \leq N$.  Furthermore, the elements in $B$ are distinct.  Indeed for $(x_1,...,x_n),(y_1,...,y_n) \in S_k$, suppose

$$\sum\limits_{i = 1}^n  x_i  (2s)^{i-1} = \sum\limits_{i = 1}^n   y_i (2s)^{i-1}. $$

\noindent For sake of contradiction, let $m$ be the largest integer such that $x_m \neq y_m$.  If $x_m - y_m > 0$, then we have

$$0 = \sum\limits_{i = 1}^m  (x_i - y_i)  (2s)^{i-1}\geq (2s)^{m -1} - \frac{s-1}{2s-1}((2s)^{m-1} - 1) > 0,$$

\noindent which is a contradiction.  An analogous argument with $x_m - y_m < 0$ gives another contradiction.  Thus we have $$|B|  \geq \frac{s^{n-2}}{n} \geq N^{1 - \frac{1}{c\sqrt{\log N}}},$$

\noindent where $c$ is an absolute constant.  Now if $B$ contained a 3-term arithmetic progression, then

$$\sum\limits_{i = 1}^n  x_i   (2s + 1)^{i-1} + \sum\limits_{i = 1}^n   z_i (2s + 1)^{i-1} = \sum\limits_{i = 1}^n 2 y_i (2s + 1)^{i-1},$$

\noindent would imply that $x_i + z_i = 2y_i$ for all $i$ by the previous argument.  Hence $(y_1,...,y_n)$ would be the midpoint of $(x_1,...,x_n)$ and $(z_1,...,z_n)$, which is a contradiction since no three members in $S_k$ are collinear.


\begin{thebibliography}{5}

\vskip.125in

\bibitem[AS74]{ajtai} M. Ajtai and E. Szemer\'edi, Sets of Lattice Points That Form No Squares, \emph{Studia. Scientiarum Mathematicarum Hungarica.} \textbf{9} (1974), 9--11.

\bibitem[Be]{Be} F.A. Behrend, On sets of integers which contain no three terms in arithmetical progression, \emph{Proc. Nat. Acad. Sci. U.S.A.} \textbf{32} (1946), 331--332.

\bibitem[CEGPSSS92]{CEGPSSS} B. Chazelle, H. Edelsbrunner, L. Guibas, R. Pollack, R.Seidel,
M. Sharir, and J. Snoeyink, {Counting and cutting cycles of lines and rods in space},
\emph{Computational Geometry: Theory and Applications}, \textbf{1} (1992), 305--323.

\bibitem[CF13]{fox} D. Conlon and J. Fox, Graph removal lemmas, \emph{Surveys in Combinatorics}, Cambridge University Press, 2013.

\bibitem[GK10]{GK1} L. Guth and N. Katz, Algebraic methods in discrete analogs of the Kakeya problem, \emph{Adv. Math.} \textbf{225} (2010), 2828--2839.


\bibitem[KSS10]{KSS} H. Kaplan, M. Sharir, and E. Shustin, {On lines and joints}, Discrete Comput Geom \textbf{44} (2010),
838--843.

\bibitem[O]{O} J.G. Oxley, \emph{Matroid Theory}, Oxford University Press, 2006.

\bibitem[Q10]{Q} R. Quilodr\'an, { The joints problem in ${\bf R^n}$}, \emph{Siam J. Discrete Math.}, \textbf{23}, 2211--2213.

\bibitem[RS78]{ruzsa} I. Ruzsa and E. Szemer\'edi, Triple systems with no six points carring three triangels, \emph{Colloq. Math. Soc. Janos Bolyai} \textbf{18} (1978), 939--945.

\end{thebibliography}
\end{document}